\documentclass[11pt,reqno]{amsart}

\usepackage[T1]{fontenc}
\usepackage{fourier,comment}
\usepackage{microtype}
\usepackage{amsmath,amssymb,amsthm,mathtools}
\usepackage{mathrsfs}
\usepackage[margin=1.05in]{geometry}
\usepackage{enumitem}
\usepackage{dutchcal}
\usepackage{upgreek}
\usepackage[mathscr]{euscript}
\usepackage[colorlinks=true,linkcolor=blue,citecolor=blue,urlcolor=blue]{hyperref}
\usepackage{graphicx}
\usepackage{float}

\newcommand{\X}{\mathbf X}

\makeatletter
\def\resetMathstrut@{%
  \setbox\z@\hbox{%
    \mathchardef\@tempa\mathcode`\(\relax
    \def\@tempb##1"##2##3{\the\textfont"##3\char"}%
    \expandafter\@tempb\meaning\@tempa \relax
  }%
  \ht\Mathstrutbox@1.2\ht\z@ \dp\Mathstrutbox@1.2\dp\z@
}
\makeatother

\newtheorem{theorem}{Theorem}[section]
\newtheorem{proposition}[theorem]{Proposition}
\newtheorem{lemma}[theorem]{Lemma}
\newtheorem{question}[theorem]{Question}
\newtheorem{problem}[theorem]{Problem}

\theoremstyle{remark}
\newtheorem{remark}[theorem]{Remark}

\newcommand{\E}{\mathbb{E}}

\newcommand{\supp}{\operatorname{supp}}
\newcommand{\1}{\mathbf{1}}

\newcommand{\Z}{\mathbb Z}
\newcommand{\R}{\mathbb R}
\newcommand{\N}{\mathbb N}

\newcommand{\TSP}{\operatorname{TSP}}

\renewcommand{\le}{\leqslant}
\renewcommand{\ge}{\geqslant}
\renewcommand{\leq}{\leqslant}
\renewcommand{\geq}{\geqslant}
\newcommand{\cc}{\mathsf{c}}

\newcommand{\sub}{\mathscr{C}}
\newcommand{\eqdef}{\stackrel{\mathrm{def}}{=}}
\newcommand{\f}{\varphi}
\newcommand{\MM}{\mathcal{M}}
\newcommand{\NN}{\mathcal{N}}
\renewcommand{\subset}{\subseteq}
\renewcommand{\Pr}{\mathbb{P}}
\newcommand{\ud}[0]{\,\mathrm{d}}
\newcommand{\fp}{\mathfrak{p}}
\newcommand{\diam}{\mathrm{diam}}
\renewcommand{\setminus}{\smallsetminus}
\newcommand{\bbX}{\mathbb{X}}

\title[Planar lamplighter is not of negative type]{Planar lamplighter is not of negative type}
\author[Antonelli G., Caputo E., Cavallucci N., Nalon L., Naor A., Wald P.]{Gioacchino Antonelli, Emanuele Caputo, Nicola Cavallucci, Luca Nalon, Assaf Naor, Pietro Wald}

\thanks{G.~A.~has been partially supported by the NSF  grant DMS-2550590. E.C. was supported by the Italian Ministry of University and Research (MUR) under the FIS 2 programme, project SingMeas: “Singular Structures in the Geometry of Measures: decompositions, rigidity and rectifiability” (project code FIS-2023-02725; CUP E53C25001800001). L.~N.~was supported by the Swiss National Science Foundation Postdoc.Mobility Fellowship	(project number P500-2235462 `\emph{Lie groups of polynomial growth}'). A.~N.~was supported  by NSF grant DMS-2453936 and a Simons Investigator award. P.~W.~was supported by the Warwick Mathematics Institute Centre
for Doctoral Training, and gratefully acknowledges funding from the University of Warwick and
the UK Engineering and Physical Sciences Research Council (Grant number: EP/W524645/1).}

\begin{document}
\begin{abstract}

The lamplighter group over the planar integer grid is proved to not be bi-Lipschitz equivalent to any metric space of negative type,  so  in particular  it does not admit a bi-Lipschitz embedding into $L_1$. This  shows the existence of finitely generated metabelian groups on which word metrics are never comparable up to constant factors to conditionally negative definite (CND) kernels, and that the property of admitting a word metric-comparable CND kernel is not preserved by wreath products.

\end{abstract}

\maketitle

\vspace{-0.2in}

\section{Introduction}

The main result herein is the following theorem:

\begin{theorem}\label{thm:MainIntro}  The metric space $(\Z_2\wr\Z^2,d_{\Z_2\wr\Z^2}^{1/2})$ does not admit a bi-Lipschitz embedding into $\ell_2$. In fact, for every integer $n\ge 2$ the Euclidean distortion of $(\Z_2\wr\Z_n^2,d_{\Z_2\wr\Z_n^2}^{1/2})$ is of order  $\sqrt{\log n}$. 
\end{theorem} 
See Section~\ref{sec:Preliminaries}  for  (basic and standard) definitions and terminology on metric embeddings and lamplighter groups. In particular,  for concreteness we  write  $d_{\Z_2\wr \Z^2}$ and $d_{\Z_2\wr \Z_n^2}$ for  the word metric on, respectively,  $\Z_2\wr \Z^2$ and $\Z_2\wr \Z_n^2$ that is induced by their standard generating sets, even though our results hold  for any choice of generating set (changing generators  can only influence implicit universal constants). 

The following problem was posed in~\cite{NaorPeres2011} (see specifically Question 10.1 of~\cite{NaorPeres2011}, as well as  the discussion of its motivation on page~57 there), and  reiterated in e.g.~\cite{BaudierMotakisSchlumprechtZsak2022,GartlandRandrianantoaninaRandrianarivony2026}:   
\begin{problem}\label{quest:NaorPeres}
    Determine whether or not $\Z_2\wr\Z^2$ admits a bi-Lipschitz embedding into $L_1$.
\end{problem}

Because $L_1$ equipped with the metric $\sqrt{\|\cdot-\cdot\|_1}$ embeds isometrically into $\ell_2$~\cite{Schoenberg1938}, Theorem~\ref{thm:MainIntro} resolves Problem~\ref{quest:NaorPeres}. In terms of distortion growth, we get the following result and open question.  For every integer  $n\ge 3$, let $\cc_1^n(\Z_2\wr\Z^2)$ denote the largest possible $L_1$ distortion of an $n$-point subset of $\Z_2\wr\Z^2$. We prove herein that  $\cc_1^n(\Z_2\wr\Z^2)\gtrsim \log\log n$,\footnote{We use the following (standard) conventions for asymptotic notation, in addition to the usual $O(\cdot), \Omega(\cdot)$ notation. Given $a, b > 0$, by writing $a\lesssim b$ or $b \gtrsim a$ we mean that $a\le \kappa b$ for some universal constant
$\kappa>0$, and $a\asymp b$ stands for $(a\lesssim b) \wedge  (b\lesssim  a)$. } and ask if this is sharp (we suspect that it is):

\begin{question}  Is it true that $\cc_1^n\big(\Z_2\wr\Z^2\big)\asymp \log\log n$ for every integer $n\ge 3$?

\end{question}

\smallskip

By~\cite{Sch35,Sch37}, a metric $d$ on a set $\MM$ has the property that $(\MM,\sqrt{d})$ embeds isometrically into a Hilbert space if and only if $d$ is a {\em conditionally negative (semi)definite} (CND) kernel on $\MM$, namely, for every $n\in \N$ and every $x_1,\ldots,x_n\in \MM$ the $n$-by-$n$ matrix $(d(x_i,x_j))$ is negative semidefinite on the orthogonal complement of the constant vectors in $\mathbb{C}^n$; the latter property for $d$ is commonly called {\em negative type}. Thus, given a finitely generated group $G$ equipped with a word metric $d$  that is induced by some finite generating set, asking if  $(G,\sqrt{d})$ admits a bi-Lipschitz embedding into a Hilbert space is equivalent to asking if there is a CND kernel on $G$ that is bounded from above and from below by positive constant multiples of $d$ (those multiples may depend on the  generating set but not on the group elements whose distance is being evaluated). We will henceforth call such groups   {\em {CND}-comparable groups}, which is legitimate  as the aforementioned property is independent of the choice of the generating set.  

   Groups that are known to be CND-comparable include: Coxeter groups~\cite{BJS88}; groups of polynomial growth~\cite{Assouad1983}; Thompson's  group F~\cite{Bur99} (see also~\cite[Theorem 1.5]{ArzhantsevaGubaSapir2006}); hyperbolic groups~\cite{BuyaloDranishnikovSchroeder2007}; wreath products of the form $G\wr H$ where $G$ has a linearly proper measured-wall metric and $H$ either has linear growth~\cite{NaorPeres2008} or is a free group~\cite{CornulierStalderValette2012}; mapping class groups~\cite{Hum17} (see also~\cite{BBF15} for a key precedent). More such examples can be found in~\cite{NR97,ArzhantsevaGubaSapir2006,HS13,Pet26}. 
  
  Non-CND-comparable groups are known to exist: the first such construction is due to~\cite{Gromov_Random} (see also its strengthening~\cite{Osajda2020}), and more examples were constructed in~\cite{ArzhantsevaDrutuSapir2009}, as well as~\cite{AustinAmenable} (see also the generalization in~\cite{OO13}) and~\cite{BrieusselZheng2021}. The examples of~\cite{Gromov_Random,ArzhantsevaDrutuSapir2009} are not amenable, \cite{AustinAmenable} constructed a non-CND-comparable amenable group which is, in fact, solvable of derived length $4$, and~\cite{BrieusselZheng2021} obtained a non-CND-comparable group that is solvable of derived length $3$.  
  
  Theorem~\ref{thm:MainIntro} says that $\Z_2\wr \Z^2$ is not CND-comparable, the first  such  group that is solvable of derived length $2$, i.e., metabelian (this is optimal as Abelian finitely generated groups are  CND-comparable).   Moreover, $\Z_2\wr \Z^2$ is the first ``classical'' group that is known not to be CND-comparable, as the aforementioned examples are  groups that were specially constructed for such purposes. One can mechanically produce more such groups from the ensuing proof of Theorem~\ref{thm:MainIntro}; see Remark~\ref{rem:MoreGeneralGroups} below.

Theorem~\ref{thm:MainIntro} is also a new demonstration of the distinction between  being a CND-comparable group and having attained Hilbert compression exponent $1/2$ (the notion of compression exponent~\cite{GK04} is recalled in Section~\ref{sec:Preliminaries}). This was known by~\cite{ArzhantsevaDrutuSapir2009} for non-amenable groups, and~\cite{BrieusselZheng2021} proved it for solvable groups of derived length  3. By~\cite{NaorPeres2011}, the Hilbert compression exponent of $\Z_2\wr \Z^2$ equals $1/2$, and the supremum in the definition of this exponent is attained, i.e., there is a Lipschitz  $f:\Z_2\wr \Z^2\to \ell_2$ satisfying $\|f(x)-f(y)\|_2\gtrsim d_{\Z_2\wr \Z^2}(x,y)^{1/2}$ for every $x,y\in \Z_2\wr \Z^2$.  At the same time,  Theorem~\ref{thm:MainIntro} says that $\Z_2\wr \Z^2$ is not CND-comparable, so the two notions can differ even for classical metabelian groups.

Finally, Theorem~\ref{thm:MainIntro}  contributes to the literature on group properties that wreath products preserve. By~\cite{CornulierStalderValette2012}, wreath products preserve the Haagerup property and coarse embeddability into Hilbert space. By~\cite{Li2010}, having a positive Hilbert compression exponent is also preserved under wreath products. Theorem~\ref{thm:MainIntro} shows that wreath products need not preserve the property of being  CND-comparable.

\subsection{On the proof of Theorem~\ref{thm:MainIntro}}\label{sec:proof sketch} Proposition~\ref{prop:quad-crossIntro} below is a Poincar\'e-type inequality that we will prove in Section~\ref{sec:cross proof}  . For its formulation, we will first introduce the following  simple notation. Given $0\le \fp \le 1$, we will let $b^\fp$ be a $\fp$-biased Bernoulli random variable, namely, $b^\fp$ is distributed over $\{0,1\}$ and it takes the value  $1$ with probability $\fp$. We will also let $b=b^{1/2}$ be a standard Bernoulli random variable, namely, it is distributed uniformly over $\{0,1\}$.  Given $n\in \N$ and $(i,j)\in \Z_n^2$, denote: 
\begin{equation}\label{eq:def random cross}
\sub_{ij}\eqdef  \big(\{i\}\times \Z_n\big)\cup \big(\Z_n\times \{j\}\big).
\end{equation}
Finally, letting $\{e_y\}_{y\in \Z_n^2}$ be the standard  basis of $\Z_2^{\Z_n^2}$, consider the following random vectors $\nu^\fp,\gamma\in \Z_2^{\Z_n^2}$:
\begin{equation}\label{eq:random vectors}
\nu^\fp\eqdef \sum_{y\in \Z_n^2}b_y^{\fp}e_y  \qquad\mathrm{and}\qquad  \gamma\eqdef \sum_{y\in \sub_{ij}}b_{y}e_{y},
\end{equation} 
where $\{b^\fp_y\}_{y\in \Z_n^2}$ are  i.i.d. copies of $b^\fp$, and  $(i,j)\in \Z_n^2$ is distributed  uniformly over $\Z_n^2$ and is independent of $\{b_y\}_{y\in \Z_n^2}$, which are   i.i.d.~copies of  $b$.  So, $\nu^\fp$ is $\fp$-Bernoulli noise, namely, it is a Boolean vector  with independent coordinates, each of which is nonzero with probability $\fp$. The random vector $\gamma$ is the indicator of the random subset of an axis-parallel cross with uniformly random center that is obtained by retaining each of its elements independently with probability $1/2$.

\begin{proposition}\label{prop:quad-crossIntro}
Every $\f:\mathbb Z_2^{\mathbb Z_n^2}\to \ell_2$ satisfies: 
\begin{equation}\label{eq:bernoulli}
\int_0^1 \E\big[\|\f(x+\nu^\fp)-\f(x) \|_2^2\big]\frac{\ud \mathfrak{p}}{\fp^{\frac32}}\lesssim n \E\big[\|\f(x+\gamma)-\f(x) \|_2^2\big].
\end{equation}
\end{proposition}

In words, the left hand side of~\eqref{eq:bernoulli} is a suitable aggregate  over $0<\fp<1$ of the expected square displacement of $\f$ when one perturbs its uniformly random input by  $\fp$-Bernoulli noise. Proposition~\ref{prop:quad-crossIntro} asserts that this aggregate displacement can never exceed  $n$ times the  expected displacement of $\f$ when one perturbs its uniformly random input by a random vector whose coordinates vanish outside the random cross in~\eqref{eq:def random cross} with its center $(i,j)$  uniformly  random, and the coordinates within it are i.i.d.~standard Bernoulli. 

The role of  the cross in Proposition~\ref{prop:quad-crossIntro} mimics its  use  in~\cite{gartlandOstrovskii2026}, though~\cite{gartlandOstrovskii2026}  treats a multi-scale union of crosses, which is a feature that we do not need herein; instead  Proposition~\ref{prop:quad-crossIntro} considers a random sparsification of a randomly shifted cross.   

It suffices to prove~\eqref{eq:bernoulli}  for characters of the Abelian group $\Gamma_n=\Z_2^{\mathbb Z_n^2}$ (an $n^2$-dimensional Boolean hypercube), since those characters form an orthonormal basis of $L_2(\Gamma_n)$ and~\eqref{eq:bernoulli} is quadratic. Checking that~\eqref{eq:bernoulli} indeed holds for the characters of $\Gamma_n$  is a  simple exact computation, carried out in Section~\ref{sec:cross proof}. Our inspiration to consider the quadratic inequality~\eqref{eq:bernoulli}  comes from~\cite{PythHyper}, where similarly working with a quadratic inequality led to a simple proof of a statement about  $L_1$ nonembeddability.

With Proposition~\ref{prop:quad-crossIntro} at hand, the proof of the second part of Theorem~\ref{thm:MainIntro} proceeds as follows. Its stated upper bound on the Euclidean distortion is known (see~\cite[Corollary 8]{BaudierMotakisSchlumprechtZsak2022} or~\eqref{eq:lp version} below), so we need to prove  the matching lower bound. Suppose that $D\ge 1$ and  $f:\Z_2\wr \Z_n^2\to \ell_2$ satisfy: 
\begin{equation}\label{eq:D distortion assumption}
\forall x,y\in \Z_2\wr\Z_n^2,\qquad d_{\Z_2\wr \Z_n^2}(x,y)^{\frac12}\le \|f(x)-f(y)\|_2\le Dd_{\Z_2\wr \Z_n^2}(x,y)^{\frac12}.
\end{equation}
The task is to deduce that $D$ must be at least a positive universal constant multiple of $\sqrt{\log n}$. This conclusion holds, in  fact,  for the restriction of $f$ to the zero section $(\Z_2\wr \Z_n^2)_0=\{(A,0):\ A\subseteq \Z_n^2\}\triangleleft\Z_2\wr \Z_n^2$, consisting of all those lamp configurations for which the lamplighter is at the identity element $0$ of $\Z_n^2$. 

Indeed, define $\f:\mathbb Z_2^{\mathbb Z_n^2}\to \ell_2$ by setting $\f(x)=f(\supp(x),0)$ for  $x=(x_y)_{y\in \Z_n^2}\in \Z_2^{\mathbb Z_n^2}$, where $\supp(x)$ is the support $\{y\in \Z_n^2: x_{y}=1\}$ of $x$. Note that this $\f$ probes the values of $f$ only  on the zero section $(\Z_2\wr \Z_n^2)_0$. A  substitution of~\eqref{eq:D distortion assumption} into~\eqref{eq:bernoulli} gives:
\begin{equation}\label{eq:substitute into Poincare}
\int_0^1 \E \Big[d_{\Z_2\wr \Z_n^2} \big((\supp(\nu^{\fp}),0),(\emptyset,0)\big)\Big]\frac{\ud \mathfrak{p}}{\fp^{\frac32}} \lesssim nD^2 \E\Big[d_{\Z_2\wr \Z_n^2}\big((\supp(\gamma),0\big),(\emptyset,0)\big)\Big].
\end{equation}

The lamplighter distances that appear in~\eqref{eq:substitute into Poincare} satisfy the following estimates: 
\begin{equation}\label{eq:the TSP computations needed}
d_{\Z_2\wr \Z_n^2} \big((\supp(\gamma),0),(\emptyset,0)\big)\lesssim n\qquad\mathrm{and}\qquad \forall \frac{1}{n^2}\le \fp\le 1,\quad \E \Big[d_{\Z_2\wr \Z_n^2} \big((\supp(\nu^\fp),0),(\emptyset,0)\big)\Big]\gtrsim n^2\sqrt{\fp}.
\end{equation}
The proof of~\eqref{eq:the TSP computations needed} is included in Section~\ref{sec:TSP computation} below; the first part of~\eqref{eq:the TSP computations needed} is deterministic and immediate to check, and the second part of~\eqref{eq:the TSP computations needed} has a quick justification. By~\eqref{eq:the TSP computations needed} the left hand side of~\eqref{eq:substitute into Poincare} satisfies: 
$$
\int_0^1 \E \Big[d_{\Z_2\wr \Z_n^2} \big((\supp(\nu^{\fp}),0),(\emptyset,0)\big)\Big]\frac{\ud \mathfrak{p}}{\fp^{\frac32}} \gtrsim n^2\int_{\frac{1}{n^2}}^1\frac{\ud\fp}{\fp}\asymp n^2\log n.
$$
At the same time,  by~\eqref{eq:the TSP computations needed} the right hand side is of~\eqref{eq:substitute into Poincare}  is  $O(D^2n^2)$. Consequently, $D\gtrsim \sqrt{\log n}$, as required.

The deduction of the first part of Theorem~\ref{thm:MainIntro} (about $\Z_2\wr \Z^2$) from its second part (about  $\Z_2\wr \Z_n^2$) is a consequence of the following lemma, whose proof appears in Section~\ref{sec:bi-embedding sec for reduction}  below:

\begin{lemma}\label{lem:embed in each} For $n\in \N$, let $\mathbb{X}_n\eqdef\{(A,0)\in (\Z_2\wr\Z^2)_0:\ A\subset [-n,n]^2\}$, equipped with the metric of $\Z_2\wr \Z^2$. Then, $(\Z_2\wr\Z_n^2)_0$ embeds with distortion $O(1)$ into $\mathbb{X}_n$. Conversely, $\mathbb{X}_n$ embeds isometrically into $(\Z_2\wr \Z_{4n}^2)_0$. 
\end{lemma}

The   antecedent~\cite{gartlandOstrovskii2026} of the reasoning herein has a higher dimensional sequel~\cite{GartlandOstrovskiiRabaniYoung2026} that uses random
measures supported on dyadic cubes to obtain bounds that are sharp also in terms of dimension.  Inspired by this, we ask if  for every $m\geq 2$, any embedding of $\Z_2\wr \Z_n^m$ into $L_1$ incurs distortion that is at least a positive universal constant multiple of $m\log n$. This would be sharp by~\cite[Corollary 8]{BaudierMotakisSchlumprechtZsak2022}.

\subsection{Further results and questions} For a discrete metric space $\MM$ and $q\ge 1$, let  $\alpha_q^*(\MM)$ denote its $L_q$ compression exponent; the definition is recalled in Section~\ref{sec:Preliminaries}. By~\cite{NaorPeres2011}, $\alpha_q^*(\Z_2\wr\Z^2)=\max\{1/q,1/2\}$, and furthermore, if $q>1$, then  $\alpha_q^*(\Z_2\wr\Z^2)$ is attained. Thus, for  $1<q\le 2$ we can define $\kappa(q)>0$ to be the supremum of those $\kappa>0$ for which there is a $1$-Lipschitz map $f:\Z_2\wr\Z^2\to L_{q}$ satisfying:
\begin{equation*}
\forall x,y\in \Z_2\wr \Z^2,\qquad \|f(x)-f(y)\|_{q}\geq \kappa d_{\Z_2\wr\Z^2}(x,y)^{\frac{1}{q}}.
\end{equation*}
By mechanically tracking the implicit constants in the proof of~\cite[Theorem 3.1]{NaorPeres2011}, one sees that $\kappa(q)\gtrsim q-1$. Thanks to Proposition~\ref{prop:quad-crossIntro}, we can now show that this result is optimal:

\begin{proposition}\label{prop:kappaepsIntro} For every $1<q\le 2$ we have $\kappa(q)\asymp q-1$.
\end{proposition}

To see how to deduce Proposition~\ref{prop:kappaepsIntro} from the new results that we described in Section~\ref{sec:proof sketch}, fix $1<q\le 2$. Suppose that $f:(\Z_2\wr\Z^2)_0\to L_{q}$ and $\kappa>0$ satisfy $\kappa d_{\Z_2\wr \Z^2}(x,y)^{1/q}\le \|f(x)-f(y)\|_q\le d_{\Z_2\wr \Z^2}(x,y)$ for every $x,y\in (\Z_2\wr \Z^2)_0$. The goal is to deduce from this that $\kappa\lesssim q-1$. By~\cite{Schoenberg1938} there is $g:L_q\to \ell_2$ such that $\|g(u)-g(v)\|_2=\|u-v\|_q^{q/2}$ for every $u,v\in L_q$. Hence, $h=g\circ f: (\Z_2\wr\Z^2)_0\to \ell_2$ satisfies $\kappa^{q/2} d_{\Z_2\wr \Z^2}(x,y)^{1/2}\le \|h(x)-h(y)\|_2\le d_{\Z_2\wr \Z^2}(x,y)^{q/2}$ for every $x,y\in (\Z_2\wr \Z^2)_0$. By composing this $h$ with the $O(1)$ distortion embedding of  Lemma~\ref{lem:embed in each}, we get that for every $n\in \N$ there is $\f_n:(\Z_2\wr\Z_n^2)_0\to \ell_2$ such that $\kappa^{q/2}d_{\Z_2\wr \Z_n^2}(x,y)^{1/2}\lesssim  \|\f_n(x)-\f_n(y)\|_2\lesssim d_{\Z_2\wr \Z_n^2}(x,y)^{q/2}$ for every $x,y\in (\Z_2\wr\Z_n^2)_0$. By combining these guarantees for $\f_n$ with~\eqref{eq:bernoulli} and~\eqref{eq:the TSP computations needed} we conclude that $\kappa^qn^2\log n\lesssim n^{1+q}$. Thus:
\begin{equation}\label{eq:kappa bound for all q}
\forall n\in \{2,3,\ldots,\},\qquad \kappa\lesssim \frac{n^{1-\frac{1}{q}}}{(\log n)^{\frac{1}{q}}}.
\end{equation}
Choosing $n\asymp e^{q/(q-1)}$ in~\eqref{eq:kappa bound for all q} minimizes (up to universal factors) its right hand side, and yields $\kappa\lesssim q-1$.

For $q\ge 1$, let $\cc_q(\MM,d_\MM)$ denote the $L_q$ distortion of a separable metric space $(\MM,d_\MM)$. Then we have: 
\begin{equation}\label{eq:lp version}
\forall 1\le q\le 2, \ \forall n\in \{2,3,\ldots\},\qquad   \cc_q\left((\Z_2\wr\Z_n^2)_0,d_{\Z_2\wr\Z_n^2}^{\frac{1}{q}}\right)\asymp (\log n)^{\frac{1}{q}}.
\end{equation}
To justify the upper bound on the distortion in~\eqref{eq:lp version}, even for embeddings of all of $\Z_2\wr \Z_n^2$ rather than only its zero section, by~\cite[Corollary~8]{BaudierMotakisSchlumprechtZsak2022} (using only the fact that $\log |\Z_n^2|\asymp \log n$) there is $f: \Z_2\wr\Z_n^2 \to L_1$ that satisfies $d_{\Z_2\wr \Z_n^2}(x,y)\lesssim \|f(x)-f(y)\|_1\lesssim (\log n)d_{\Z_2\wr \Z_n^2}(x,y)$ for all $x,y\in \Z_2\wr \Z_n^2$. By~\cite{BDCK1966} (see also the treatment in~\cite{WW75} or~\cite[Section~3]{NaorProceedings}), there is $\psi:L_1\to L_q$ satisfying $\|\psi(u)-\psi(v)\|_q=\|u-v\|_1^{1/q}$ for every $u,v\in L_1$. The composition $\psi\circ f$ exhibits the upper bound on the distortion in~\eqref{eq:lp version}.  

The lower bound on the distortion in~\eqref{eq:lp version} is deduced as follows from Theorem~\ref{thm:MainIntro}, similarly to how we justified Proposition~\ref{prop:kappaepsIntro}.  If $\f:(\Z_2\wr\Z_n^2)_0\to L_q$ satisfies $d_{\Z_2\wr \Z_n^2}(x,y)^{1/q}\lesssim \|\f(x)-\f(y)\|_q\lesssim Dd_{\Z_2\wr \Z_n^2}(x,y)^{1/q}$ for some $D\ge 1$ and all $x,y\in (\Z_2\wr \Z_n^2)_0$, then by composing $f$ with the aforementioned $g:L_q\to \ell_2$ that satisfies $\|g(u)-g(v)\|_2=\|u-v\|_q^{q/2}$ for every $u,v\in L_q$, we see that $((\Z_2\wr \Z_n^2)_0,d_{\Z^2\wr \Z_n^2}^{1/2}\big)$ embeds with distortion $D^{q/2}$ into $\ell_2$. By the proof of Theorem \ref{thm:MainIntro}, we get $D^{q/2}\gtrsim \sqrt{\log n}$, i.e., $D\gtrsim (\log n)^{1/q}$, as required.

We do not know~\eqref{eq:lp version} for $2<q<\infty$. In fact, we do not know whether  for any  $2<q<\infty$ the distortion in~\eqref{eq:lp version} is $O(1)$. Furthermore, Question~\ref{Q:snowflake of hilbert} below about embeddings into Hilbert space remains open:

\begin{question}\label{Q:snowflake of hilbert} Is there any  $0<\theta <\frac12$ such that $\big((\Z_2\wr \Z^2)_0,d^\theta_{\Z_2\wr \Z^2}\big)$ admits a bi-Lipschitz embedding into $\ell_2$? 
\end{question}

The restriction $0<\theta <1/2$ in Question~\ref{Q:snowflake of hilbert} was made because by combining~\eqref{eq:lp version} with Lemma~\ref{lem:embed in each} and the fact that   $\ell_2$ embeds isometrically into $L_q$ (see e.g.~\cite[Proposition 6.4.2]{AlbiacKalton}), the same question has a negative answer when $1/2\le \theta \le 1$. 

The following question is about the well-studied  algorithmic notion of {\em sketchability}~\cite{AMS96,SS02}  (see also the discussion in~\cite{AKR15,KN21}) for the planar lamplighter group, which is equivalent to sketchability of traveling salesman tours, per the standard connection that is recalled in Section~\ref{sec:Preliminaries}.  It arises naturally here because a positive answer to Question~\ref{Q:snowflake of hilbert} would imply that it too has a positive answer, as explained in~\cite{AKR15,KN21}, relying on the important works~\cite{KOR99,IM99}. 

\begin{question}\label{Q:sketching} Is $(\Z_2\wr \Z^2)_0$ sketchable?
\end{question}

In \cite{NaorPeres2011} it is proved that:
\begin{equation}\label{eq:zer0 section}
\forall 1\le q<\infty,\qquad \alpha_q^*\big((\Z\wr\Z)_0\big)=\max\left\{\frac{q+1}{2q},\frac{3}{4}\right\}.
\end{equation}
Section~\ref{sec:new zer section compression} below builds (nontrivially) on the method of~\cite{NaorPeres2011} to show that the $L_q$ compression exponent of the zero-section of $\Z_2\wr\Z^2$ also equals the right hand side of~\eqref{eq:zer0 section}. In fact, we obtain the following result:
\begin{theorem}\label{thm:LpcompressionIntro}
    For every $m\in \N$ and every $1\le q<\infty$ we have:
\begin{equation}\label{eq:proposed-exponentintro}
\alpha_q^*\big((\Z_2\wr\Z^m)_0\big)=\max\left\{\frac{q+m-1}{mq},\frac{m+1}{2m}\right\}.
\end{equation}
\end{theorem}

Theorem~\ref{thm:LpcompressionIntro} is not formally related to the new results that we discussed thus far, except that the realization (through how Theorem~\ref{thm:MainIntro}  is proved) that the zero section of $\Z_2\wr \Z^2$ turns out to be relevant to Problem~\ref{quest:NaorPeres} leads one to naturally wonder about the question that Theorem~\ref{thm:LpcompressionIntro}  answers.

\bigskip

\noindent\textbf{AI disclosure.}  An LLM had the following major role in  this article.  A.N. delivered a minicourse titled ``Distortion Growth'' at the 14th School on Analysis and Geometry in Metric Spaces (Trento, June 2026), where he recalled Problem~\ref{quest:NaorPeres} because he suspected that the method of the breakthrough~\cite{gartlandOstrovskii2026} should be relevant to it.  G.A., E.C., N.C., L.N., P.W. asked (not involving A.N.) ChatGPT-5.5 Pro for assistance. It suggested an argument  which indeed cleverly considers steps and objects that are in the spirit of~\cite{gartlandOstrovskii2026}. After G.A., E.C., N.C., L.N., P.W. wrote the resulting proof, they sent it to A.N., who proceeded by building on that approach  to obtain the stronger (yet simpler to prove) result herein on the distance of $(\Z_2\wr\Z_n^2)_0$  from negative type, as well as to suggest the rest of the contents of this article (except for the case $m\ge 3$ of Theorem~\ref{thm:LpcompressionIntro}). 
Even though the present article does not include any of the output of the LLM (which was only about the $L_1$ case and relied on concepts that are not needed herein), the present work is very much reliant on and catalyzed by that output, which itself builds on~\cite{gartlandOstrovskii2026}.

\section{Notation and terminology}\label{sec:Preliminaries}
Here we will recall basic notation and terminology that is needed in this article. Given a metric space $(\MM,d_\MM)$ and a normed space $(\X,\|\cdot\|_\X)$, the $(\X,\|\cdot\|_\X)$-distortion of $(\MM,d_\MM)$, denoted $\cc_{(\X,\|\cdot\|_\X)}(\MM,d_\MM)$ or simply $\cc_\X(\MM)$ when the metrics are clear from the context, is the infimum over those $D\ge 1$ for which there exists $f:\MM\to \X$ satisfying $d_\MM(x,y)\le \|f(x)-f(y)\|_\X\le Dd_\MM(x,y)$   for every $x,y\in \MM$. If no such embedding exists, then one sets $\cc_\X(\MM)=\infty$. When $\MM$ is separable and $\X=L_p=L_p(\R)$ for some $1\le p\le \infty$, we write $\cc_\X(\MM)=\cc_p(\MM)$. The invariant $\cc_2(\MM)$ is called the Euclidean distortion of $\MM$.  

For every $n\in \N$ write $\cc_\X^n(\MM)=\sup\{\cc_\X(A):\ A\subset \MM\ \ \mathrm{and}\ \ |A|\le n\}$. The sequence  $\{\cc_\X^n(\MM)\}_{n=1}^\infty$ is called the $\X$-distortion growth sequence of $\MM$.

When $(\MM,d_\MM)$ is a countable discrete metric space and $1\le p\le \infty$, its $L_p$ compression exponent, denoted $\alpha_p^*(\MM,d_\MM)$ or simply    $\alpha_p^*(\MM)$ when the metric is clear from the context,  is the supremum over those $\alpha\ge 0$ for which there is a Lipschitz function $f:\MM\to L_p$ and a constant $\kappa>0$ such that $\|f(x)-f(y)\|_p\ge \kappa d_\MM(x,y)^\alpha$ for every $x,y\in \MM$. We say that this compression exponent is attained if the above supremum is a maximum. 

Let $(\MM,d_\MM)$ be a metric space. Given a finite subset $A$ of $\MM$ and $x,y\in \MM$,  we let $\TSP^{d_\MM}(A;x,y)$ denote the infimal $d_\MM$-length of  a traveling salesman (TSP) tour that starts at $x$, visits all of $A$, and ends at $y$, i.e., if we write $k=|A\setminus\{x,y\}|$ and let $\{a_1,\ldots,a_k\}=A\setminus\{x,y\}$ be an arbitrary enumeration of $A\setminus\{x,y\}$, then:
\begin{equation}\label{eq:permutational definition}
\TSP^{d_\MM}(A;x,y)\eqdef \min_{\pi\in S_k} \bigg(d_\MM(x,a_{\pi(1)})+\sum_{i=1}^{k-1} d_\MM(a_{\pi(i)},a_{\pi(i+1)})+d_\MM(a_{\pi(k)},y)\bigg),
\end{equation}
where $S_k$ denotes the permutations of $\{1,\ldots,k\}$, and  $\TSP^{d_\MM}(A;x,y)=d_\MM(x,y)$ when $A\subset \{x,y\}$. By dropping the first summand in~\eqref{eq:permutational definition}, we record for later use the following basic lower bound:
\begin{equation}\label{eq:trivial lower}
\TSP^{d_\MM}(A;x,y)\ge \sum_{a\in A\setminus \{x,y\}} d_\MM\big(a, (A\setminus\{a\})\cup\{y\}\big)\ge \sup_{r>0} r \big|\big\{a\in A\setminus  \{x,y\}:\ d_\MM\big(a, (A\setminus\{a\})\cup\{y\}\big)\ge r\big\}\big|.
\end{equation}
Another immediate consequence of~\eqref{eq:permutational definition} is:
\begin{equation}\label{eq:TSP to support size}
\TSP^{d_\MM}(A;x,y)\ge (|A|-1)\inf_{\substack{u,v\in \MM\\ u\neq v}} d_\MM(u,v). 
\end{equation}

The lamplighter group $\Z_2\wr G$ over a group $G$ consists of all the pairs $(A,x)$, where $A$ is a finite subset of $G$ and $x\in G$, and equipped with the group law $(A,x)(B,y)=(A\triangle xB,xy)$. It is useful to interpret $(A,x)$ as describing a configuration where the lamplighter is at $x$, and the lamps are on at elements of $A$. The zero section $(\Z_2\wr G)_0$ of the lamplighter group over $G$ is the normal subgroup of $\Z_2\wr G$ that consists of all those $(A,x)\in  \Z_2\wr G$ for which $x=e_G$, where $e_G$ is the identity element of $G$. 

Suppose that $G$ is generated by a finite symmetric set $S\subseteq G$, and let $d_S$ denote the left-invariant word metric that is induced by $S$. The standard generating set of $\Z_2\wr G$ consists of $(\emptyset,s)$ for every $s\in S$, as well as $(\{e_G\},e_G)$. The word metric that is induced by this set of generators is denoted $d^S_{Z_2\wr G}$, and it is given by:
\begin{equation}\label{eq:lamp to TSP}
\forall (A,x),(B,y)\in \Z_2\wr G,\qquad d^S_{\Z_2\wr G}\big((A,x),(B,y)\big)=|A\triangle B|+\TSP^{d_S}(A\triangle B;x,y). 
\end{equation}
See e.g.~\cite{Genevois2025Lamplighters} for the derivation of the standard identity~\eqref{eq:lamp to TSP}. Because $d_S$ takes values in $\N\cup\{0\}$, by combining~\eqref{eq:lamp to TSP} and~\eqref{eq:TSP to support size} we see that every $(A,x),(B,y)\in \Z_2\wr G$ with $A\neq B$ satisfy:
\begin{equation}\label{eq:lamp to TSP zero section}
 \TSP^{d_S}(A\triangle B;x,y)+1\le d^S_{\Z_2\wr G}\big((A,x),(B,y)\big)\le  2\TSP^{d_S}(A\triangle B;x,y)+1. 
\end{equation}

When $G=\Z_n^m$ or $G=\Z^m$ for some $m,n\in \N$, we will always assume that the generating set $S$ is $\{\pm e_1,\ldots,\pm e_m\}$, where $e_1,\ldots,e_m$ is the standard (coordinate) basis. The metrics on  these groups we will be denoted $d_{\Z_m^n}, d_{\Z^m}$. Correspondingly, the lamplighter groups over these groups will always be assumed to be generated by the canonical generating set, and we will drop the superscripts $S$ and $d_S$ in the notation~\eqref{eq:lamp to TSP} for their metrics.

\section{Proof of Proposition~\ref{prop:quad-crossIntro}}\label{sec:cross proof}

Fix $n\in \N$. The characters of  $\Z_2^{\Z_n^2}$  are (see e.g.~\cite{ODonnell2014}) the Walsh functions $\{W_{A}\}_{A\subset \Z_n^2}$, where:
\begin{equation*}\label{eq:walsh remind}
\forall A\subset \Z_n^2,\ \forall x=(x_{y})_{y\in  \Z_n^2}\in \Z_2^{\Z_n^2},\qquad W_A(x)\eqdef (-1)^{\sum_{y\in A}x_y}. 
\end{equation*}
Fix $A\subset \Z_n^2$. The  fact that $W_A$ is a product over coordinates, combined with the independence of the coordinates of the  random variable $\nu^\fp$ that is given in~\eqref{eq:random vectors}, yields the following standard identity:
\begin{equation}\label{eq:p bernoulli noise walsh}
\E[W_A(\nu^\fp)]= \E\Big[\prod_{y\in A} (-1)^{b_y^\fp}\Big]= \Big(\E \big[(-1)^{b^\fp}\big]\Big)^{|A|}=(1-2\fp)^{|A|}.
\end{equation}
For the same reason, by the independence of the coordinates of the  random vector $\gamma$  given in~\eqref{eq:random vectors}, and because they are independent from the random variable $(i,j)$ that is distributed uniformly over $\Z_n^2$, if we let $A_1$ and $A_2$ be the projections of $A$ onto, respectively, the first and second coordinates, i.e.,
\begin{equation}\label{eq:def projections}
 A_1=\big\{s\in \Z_n:\ \big(\{s\}\times \Z_n\big)\cap A\neq\emptyset\big\}\qquad \mathrm{and}\qquad A_2=\big\{s\in \Z_n:\ \big( \Z_n\times \{s\}\big)\cap A\neq\emptyset\big\}, 
\end{equation}
then, recalling the definition of the random cross $\sub_{ij}\subset \Z_n^2$ in~\eqref{eq:def random cross}, we also have the following identity:  
\begin{align}\label{eq:rhs identity}
\begin{split}
\E[W_A(\gamma)]= \E\Big[\prod_{y\in A\cap \sub_{ij}} (-1)^{b_y}\Big]= \Pr[A\cap \sub_{ij}= \emptyset]=\Pr[i\notin A_1\ \mathrm{and}\  j\notin A_2]=\Big(1-\frac{|A_1|}{n}\Big)\Big(1-\frac{|A_2|}{n}\Big).
\end{split}
\end{align}

As   $\{W_{A}\}_{A\subset \Z_n^2}$ is an orthonormal basis of $L_2(\Z_2^{\Z_n^2})$, the quadratic nature of the desired inequality~\eqref{eq:bernoulli} means that it suffices to prove that~\eqref{eq:bernoulli} holds when $\f=W_A$ for some $A\subset \Z_n^2$, for which~\eqref{eq:bernoulli} becomes:
\begin{equation}\label{specialize to walsh}
  2\int_0^1\big(1-\E[W_A(\nu^\fp)]\big)\frac{\ud \fp}{\fp^{\frac32}}=\int_0^1\E \big[(W_A(\nu^\fp)-1)^2\big]\frac{\ud \fp}{\fp^{\frac32}}\lesssim n \E \big[(W_A(\gamma)-1)^2\big]=2n\big(1-\E[W_A(\gamma)]\big),
\end{equation}
where the first and last equalities in~\eqref{specialize to walsh} hold since $(W_A-1)^2=2(1-W_A)$ point-wise, as $W_A$ takes values in $\{-1,1\}$. 
To prove that~\eqref{specialize to walsh} indeed holds,  fix  $A\subset \Z_n^2$ and evaluate its left hand side using~\eqref{eq:p bernoulli noise walsh} as follows:  
\begin{equation}\label{eq: left integral}
2\int_0^1\big(1-\E[W_A(\nu^\fp)]\big)\frac{\ud \fp}{\fp^{\frac32}}\stackrel{\eqref{eq:p bernoulli noise walsh}}{=}2\int_0^1 \frac{1-(1-2\fp)^{|A|}}{\fp^{\frac32}}\ud \fp\asymp \int_0^1 \frac{\min\big\{\fp|A|,1\big\}}{\fp^{\frac32}}\ud \fp\asymp \sqrt{|A|}.
\end{equation}
By~\eqref{eq:rhs identity}, the right hand side of~\eqref{specialize to walsh} is $2n(1-(1-|A_1|/n)(1-|A_2|/n))$, and it is elementary to check that this quantity is at least $2\max\{|A_1|,|A_2|\}$. But, $A\subset A_1\times A_2$ by~\eqref{eq:def projections}, whence $|A|\le |A_1|\cdot |A_2|\le (\max\{|A_1|,|A_2|\})^2$. The right hand side of~\eqref{specialize to walsh} is therefore at least $2\sqrt{|A|}$, which is the required bound~\eqref{specialize to walsh} thanks to~\eqref{eq: left integral}.  \qed

\begin{remark} A more careful examination of the above reasoning shows that~\eqref{eq:bernoulli} holds with the implicit constant equal to $2\sqrt{2\pi}$, and this is optimal. That constant is immaterial herein, so we omit the details. 
\end{remark}

\section{Proof of~\eqref{eq:the TSP computations needed}}\label{sec:TSP computation}

We already explained in Section~\ref{sec:proof sketch} how  Theorem~\ref{thm:MainIntro} follows  by combining~\eqref{eq:the TSP computations needed} with Proposition~\ref{prop:quad-crossIntro}. Here we will justify~\eqref{eq:the TSP computations needed}. For its first part, observe that $\supp(\gamma)\subset \sub_{ij}$, whence: 
$$
d_{\Z_2\wr \Z_n^2} \big((\supp(\gamma),0),(\emptyset,0)\big)\stackrel{\eqref{eq:lamp to TSP zero section}}{\lesssim} 1+\TSP(\sub_{ij};0,0)\lesssim n,
$$
where the last step holds due to the following TSP tour of length $O(n)$: start at $0$, go to $(i,j)$ along a shortest path in the Cayley graph of $\Z_n^2$, perform  four return excursions from the center $(i,j)$ of the cross $\sub_{ij}$ to its four corners, and conclude by  returning to $0$ along a shortest Cayley path.  

For the rest of~\eqref{eq:the TSP computations needed}, Lemma~\ref{lem:TSP lower immediate}  below is a simple lower bound on the expected TSP cost of Bernoulli subsets of a  metric space $(\MM,d_\MM)$. In it, we write $B_{d_\MM}$ and $B_{d_\MM}^\circ$  for, respectively, closed and open balls. 

\begin{lemma}\label{lem:TSP lower immediate} Suppose that $(\MM,d_\MM)$ is a finite metric space. For $0\le \fp\le 1$ let $A_\fp$ be the random subset of $\MM$ that is obtained by including independently each element of $\MM$ with probability $\fp$. Then:
$$
\forall x,y\in \MM,\qquad \E\big[\TSP^{d_\MM}(A_\fp;x,y)\big]\ge \fp \sup_{r>0} \bigg( r\sum_{a\in \MM\setminus (B_{d_\MM}^\circ(y,r)\cup\{x\})}(1-\fp)^{|B_{d_\MM}^\circ(a,r)|-1}\bigg).
$$
Consequently, if we denote $\mathsf{V}^{d_\MM}(r)\eqdef \max_{z\in \MM} |B_{d_\MM}(z,r)|$ for every $r\ge 0$, 
then:
\begin{equation}\label{eq:with upper and lower growth}
\forall x\in \MM,\ \forall r>0,\qquad \E\big[\TSP^{d_\MM}(A_\fp;x,x)\big]\ge \fp r\big(|\MM|-\mathsf{V}^{d_\MM}(r)\big)(1-\fp)^{\mathsf{V}^{d_\MM}(r)-1}. 
\end{equation}
\end{lemma}

\begin{proof} For  $A\subset \MM$, a point $a\in \MM$ belongs to the subset  in the right hand side of~\eqref{eq:trivial lower} if and only if it differs from $x$, does not belong to $B^\circ_{d_\MM}(y,r)$, belongs to $A$, and no point of $A$ belongs to $B_{d_\MM}^\circ(a,r)\setminus\{a\}$. Hence:
\begin{align*}
\E\big[\TSP^{d_\MM}(A_\fp;x,y)\big]&\ge \sup_{r>0} \E \bigg[\sum_{a\in \MM\setminus (B_{d_\MM}^\circ(y,r)\cup\{x\})}r\1_{\{a\in A_\fp\ \mathrm{and}\ A_\fp\cap (B_{d_\MM}^\circ(a,r)\setminus\{a\})=\emptyset\}}\bigg]\\&=\sup_{r>0}  \bigg(\sum_{a\in \MM\setminus (B_{d_\MM}^\circ(y,r)\cup\{x\})}r\Pr \Big[a\in A_\fp\ \mathrm{and}\ A_\fp\cap (B_{d_\MM}^\circ(a,r)\setminus\{a\})=\emptyset\Big]\bigg)\\&= \sup_{r>0}  \bigg(\sum_{a\in \MM\setminus (B_{d_\MM}^\circ(y,r)\cup\{x\})}r\fp(1-\fp)^{|B_{d_\MM}^\circ(a,r)|-1}\bigg). \tag*{\qedhere}
\end{align*}
\end{proof}

Proposition~\ref{prop:upper growth} below will be deduced straightforwardly  from  Lemma~\ref{lem:TSP lower immediate}. It  generalizes   the second part of~\eqref{eq:the TSP computations needed}, because in its special case $G=\Z_n^2$, with $m=2$ and $\phi,C$ universal constants,   it yields the following lower bound when $\frac{1}{n^2}\leq\fp\leq 1$:
$$
\E\big[\TSP^{d_{\Z_n^2}}(\supp(\nu^\fp);0,0)\big]\gtrsim n^2\sqrt{\fp}.
$$
From here the second part of~\eqref{eq:the TSP computations needed} follows because:
$$
\E \Big[d_{\Z_2\wr \Z_n^2} \big((\supp(\nu^\fp),0),(\emptyset,0)\big)\Big]\stackrel{\eqref{eq:lamp to TSP} }{=}\E\big[|\supp(\nu^\fp)|\big]+\E\big[\TSP^{d_{\Z_n^2}}(\supp(\nu^\fp);0,0)\big]\gtrsim n^2\fp+n^2\sqrt{\fp}\asymp n^2\sqrt{\fp}. 
$$

\begin{proposition}\label{prop:upper growth} Fix $0<\phi\le 1\le C$ and $m\in \N$.  Let $G$  be a finite group whose unit element is $e$, and let $S \subset G$ be a symmetric generating set of $G$. Denote the  corresponding  left-invariant word metric by $d_S$. Assume  the following nondegeneracy requirement from diameter $\diam_S(G)$ of $G$ with respect to $d_S$:
\begin{equation}\label{eq:diam lower}
\diam_S(G)\ge \Big(\frac{|G|}{2C}\Big)^{\frac{1}{m}}.
\end{equation}
Denoting $B_S(e,r)=B_{d_S}(e,r)$ for  $r\ge 0$, assume furthermore the following upper growth bound for $d_S$-balls:
\begin{equation}\label{eq:upper growth}
 \forall r\in \big \{1,\ldots,\lfloor \phi\diam_S(G)\rfloor\big\},\qquad  |B_S(e,r)|\le C r^m.
\end{equation}
For $0\le \fp\le 1$, let $A_\fp$ be the random subset of $G$ that is obtained by including independently each element of $G$ with probability $\fp$. Then:
\begin{equation}\label{eq:lower TSP group}
\forall \frac{1}{|G|}\le \fp\le 1,\qquad \E\big[\TSP^S(A_\fp;e,e)\big]\gtrsim \frac{\phi}{\sqrt[m]{C}}|G|\fp^{1-\frac{1}{m}}. 
\end{equation}
\end{proposition}

\begin{proof} The point-wise identity  $\TSP^S(A_\fp;e,e)\ge (|A_\fp\setminus\{e\}|+1)\1_{\{A_\fp\setminus\{e\}\neq\emptyset\}}=|A_\fp\setminus\{e\}|+1-\1_{\{A_\fp\subset \{e\}\}}$ holds because the minimum nonzero $d_S$-distance is $1$. Hence, for every $0\le \fp\le 1$ we have:
\begin{equation}\label{eq:expected size bernoulli}
 \E\big[\TSP^S(A_\fp;e,e)\big]\ge \E\big[|A_\fp\setminus\{e\}|\big]+1-\Pr\Big[\bigcap_{x\in G\setminus \{e\}} \{x\notin A_\fp\}\Big]=\fp(|G|-1)+1-(1-\fp)^{|G|-1}\ge \fp|G|, 
\end{equation}
where the last step of~\eqref{eq:expected size bernoulli} holds because~\eqref{eq:diam lower} implies that  $|G|\ge 2$. Fix $1/|G|\le\fp\le 1$.  Note that~\eqref{eq:expected size bernoulli} gives~\eqref{eq:lower TSP group} if $\phi/\sqrt[m]{2C\fp}<1$, so we may assume that  $\phi/\sqrt[m]{2C\fp}\ge 1$ and set:
\begin{equation}\label{eq:choose the r}
r\eqdef \left\lfloor \frac{\phi}{\sqrt[m]{2C\fp}}\right\rfloor\in\N.
\end{equation}
 This choice of $r$  satisfies the following a priori bounds: 
\begin{equation}\label{eq:r a priori}
\frac{\phi}{2\sqrt[m]{2C\fp}}\le r\le \frac{\phi}{\sqrt[m]{2C\fp}}\le \phi\Big(\frac{|G|}{2C}\Big)^{\frac{1}{m}}\stackrel{\eqref{eq:diam lower}}{\le} \phi\diam_S(G),
\end{equation}
where the first step of~\eqref{eq:r a priori} uses the assumption $\phi/\sqrt[m]{2C\fp}\ge 1$, and the third step of~\eqref{eq:r a priori} uses $\fp\ge 1/|G|$. 

The upper bound on $r$ in~\eqref{eq:r a priori} shows that we may apply~\eqref{eq:upper growth} with the left-invariance of $d_S$ to get:
\begin{equation}\label{eq:V upper m}
\mathsf{V}^{d_S}(r)\le Cr^m, 
\end{equation} 
where we recall the notation for $\mathsf{V}^{d_\MM}(r)$ in Lemma~\ref{lem:TSP lower immediate}. Now, apply   Lemma~\ref{lem:TSP lower immediate} to conclude as follows: 
\begin{multline*}\label{eq:for upper growth}
\E\big[\TSP^S(A_\fp;e,e)\big]\stackrel{\eqref{eq:with upper and lower growth}\wedge \eqref{eq:V upper m}}{\ge} \fp r \big(|G|-C r^m\big)(1-\fp)^{C r^m-1}\\\stackrel{\eqref{eq:r a priori}}{\ge} \fp \frac{\phi}{2\sqrt[m]{2C\fp}}\left(|G|-\frac{\phi^m}{2\fp}\right)(1-\fp)^{\frac{\phi^m}{2\fp}} \ge \fp \frac{\phi}{2\sqrt[m]{2C\fp}}\left(|G|-\frac{\phi^m}{2}|G|\right)e^{-\phi^m} \asymp \frac{\phi}{\sqrt[m]{C}}|G|\fp^{1-\frac{1}{m}}, 
\end{multline*}
where the penultimate step uses $\fp\ge 1/|G|$ and $1-\fp\ge e^{-2\fp}$, which holds because  $0<\phi\le 1\le C$ and we are assuming that $\fp\le \phi^m/(2C)\le 1/2$, and the final step holds because $0<\phi\le 1$.
\end{proof}

\section{Proof of Lemma~\ref{lem:embed in each}}\label{sec:bi-embedding sec for reduction} 

For  $k\in \N$ let $q_k:\Z\to \Z_k$ be the  quotient  (reduction modulo $k$), i.e., $q_k(x)=x+k\Z$ for $x\in \Z$.  
Fix $n\in \N$.  The second part of Lemma~\ref{lem:embed in each} follows  by considering $q_{4n}^{\otimes 2}:\Z^2\to \Z_{4n}^2$, i.e., $q_{4n}^{\otimes 2}(i,j)=(q_{4n}(i),q_{4n}(j))$ for  $(i,j)\in \Z^2$, and lifting it to  $Q_{4n}:(\Z_2\wr \Z^2)_0\to (\Z_2\wr \Z_{4n}^2)_0$   by $Q_{4n}(A,0)=(q_{4n}^{\otimes 2}(A),0)$ for  $A\subset \Z^2$.  Note that $q_{4n}^{\otimes 2}$ is an isometry  of $\{-n,\ldots,n\}^2$ into $\Z_{4n}^2$ as  for $x,y\in [-n,n]^2$ all of the coordinates of $x-y$ are in $[-2n,2n]$. Hence, $q^{\otimes 2}_{4n}(A)\triangle q_{4n}^{\otimes 2}(B)= q^{\otimes 2}_{4n}(A\triangle B)$ and $|q^{\otimes 2}_{4n}(A)\triangle q^{\otimes 2}_{4n}(B)|=|A\triangle B|$ for  $A,B\subset [-n,n]^2$ (by the injectivity of $q_{4n}^{\otimes 2}$), and furthermore    $\TSP(q^{\otimes 2}_{4n}(A)\triangle q_{4n}^{\otimes 2}(B);0,0)=\TSP(q^{\otimes 2}_{4n}(A\triangle B);0,0)=\TSP(A\triangle B;0,0)$ follows immediately from the isometric property and the definition~\eqref{eq:permutational definition}. Thanks to~\eqref{eq:lamp to TSP}, these observations show that $Q_{4n}$ is an isometric embedding of $\bbX_n$ into $(\Z_2\wr \Z_{4n}^2)_0$.

The more substantial direction of Lemma~\ref{lem:embed in each} is to show that $(\Z_2\wr\Z_n^2)_0$ embeds with distortion $O(1)$ into $\bbX_n$.  For  $i\in \Z_n$, let  $\lambda_n(i)$ be the unique integer in $[-\lfloor n/2\rfloor,\lceil n/2\rceil-1]$ for which $q_n\circ \lambda_n(i)=i$. Consider the map $\lambda_n^{\otimes 2}:\Z_n^2\to \Z^2$, i.e., $\lambda_n^{\otimes 2}(i,j)=(\lambda_n(i),\lambda_n(j))$ for $(i,j)\in \Z_n^2$, and lift it to $\Lambda_n:(\Z_2\wr\Z_n^2)_0\to (\Z_2\wr\Z^2)_0$ by $\Lambda_n(A,0)=(\lambda_n^{\otimes 2}(A),0)$ for $A\subset \Z_n^2$. Then, $\Lambda_n$ takes values in $\bbX_n$ (in fact, its image is in $\bbX_{\lfloor n/2\rfloor}$). As above, since $\lambda_n^{\otimes 2}$ is (by design) a bijection, it suffices to prove that $\TSP(\Lambda_n(S);0,0)\asymp \TSP(S;0,0)$ for every $S\subset \Z_n^2$.  By considering the image under $q_n^{\otimes 2}:\Z^2\to \Z_n^2$ of a  TSP tour in $\Z^2$ starting from $0$, ending at $0$, and covering $\Lambda_n(S)$, we see that $\TSP(S;0,0)\le  \TSP(\Lambda_n(S);0,0)$. The crux of the matter is thus to prove that $\TSP(\Lambda_n(S);0,0)\lesssim \TSP(S;0,0)$. We will next explain why this is so via a short combinatorial argument. 

Write $\ell=\TSP(S;0,0)$ and let  $x_0,\ldots,x_\ell\in \Z_n^2$ be such that $x_0=x_\ell=0$ and $\{x_0,\ldots,x_\ell\}\supseteq S$, and $\{x_{i-1},x_i\}$ is an edge in the Cayley graph of $\Z_n^2$ for every $i\in \{1,\ldots,\ell\}$. It is convenient to set $x_{-1}=x_{\ell+1}=0$. We call $i\in \{1,\ldots,\ell\}$ a {\em transition} if $\{\lambda_n^{\otimes 2}(x_{i-1}),\lambda_n^{\otimes 2}(x_{i})\}$ is not an edge of the integer grid $\Z^2$. By the definition of $\lambda_n$, this forces at least one of the coordinates   of $x_{i-1}$ and $x_i$ to belong to $\{-\lfloor n/2\rfloor,\lceil n/2\rceil-1\}$, so because the above TSP tour starts and ends at $0$, if there exists a transition, then necessarily $\TSP(S;0,0)=\ell\gtrsim  n$. 

If there are no transitions, then $\lambda_n^{\otimes 2}(x_0),\ldots,\lambda_n^{\otimes 2}(x_\ell)$ forms a valid TSP tour in $\Z^2$ that starts and ends at $0$ and covers $\lambda_n^{\otimes 2}(S)$, whence in this case $\TSP(\Lambda_n(S);0,0)\le \ell=\TSP(S;0,0)$. 

It remains to treat the case in which there are transitions, so in particular $\TSP(S;0,0)\gtrsim n$, as we noted above. Let $\mathsf{E}$ be the set of edges of $\Z^2$ consisting of those $\{\lambda_n^{\otimes 2}(x_{i-1}),\lambda_n^{\otimes 2}(x_{i})\}$ for which  $i\in \{1,\ldots,\ell\}$ is not a transition. Every connected component $C$ of  the graph $\mathsf{G}=(\{\lambda_n^{\otimes 2}(x_0),\ldots,\lambda_n^{\otimes 2}(x_\ell)\},\mathsf{E})$  intersect the boundary of $[-\lfloor n/2\rfloor,\lceil n/2\rceil-1]^2$. Indeed, otherwise  $C\subset (-\lfloor n/2\rfloor,\lceil n/2\rceil-1)^2$, so if $i\in \{1,\ldots,\ell\}$ is such that $\lambda_n^{\otimes 2}(x_i)\in C$, then $\{\lambda_n^{\otimes 2}(x_{i-1}),\lambda_n^{\otimes 2}(x_{i})\}\in \mathsf{E}$, whence also  $\lambda_n^{\otimes 2}(x_{i-1})\in C$, and by iterating this we see that $C=\{\lambda_n^{\otimes 2}(x_0),\ldots,\lambda_n^{\otimes 2}(x_\ell)\}$ and there are no transitions, contrary to our assumption. This observation shows  that if we let  $\mathsf{F}$ be the set the edges of $\Z^2$ that are contained in $\partial [-\lfloor n/2\rfloor,\lceil n/2\rceil-1]^2$, then the graph $\mathsf{H}= (\{\lambda_n^{\otimes 2}(x_0),\ldots,\lambda_n^{\otimes 2}(x_\ell)\}\cup (\Z^2\cap \partial [-\lfloor n/2\rfloor,\lceil n/2\rceil-1]^2),\mathsf{E}\cup\mathsf{F})$ is connected. As $|\mathsf{F}|\asymp n$, the graph $\mathsf{H}$ has $O(|\mathsf{E}|+n)=O(\ell+n)\asymp \ell$ edges, so it has a walk of size $O(\ell)$ that visits all of its vertices (see e.g.~\cite[exercise~1.9]{BM26} for this standard fact, which can be seen by taking a spanning tree of $H$ and traversing  it using depth-first search, or doubling the edges of $\mathsf{H}$ to get a graph all of whose vertices have even degrees, whence  it has a Eulerian circuit). We therefore have $\TSP(\Lambda_n(S);0,0)\lesssim \ell\lesssim \TSP(S;0,0),$ as required.\qed

\begin{remark}\label{rem:MoreGeneralGroups}
Let $G,H$ be infinite finitely generated groups with identity elements $e_G,e_H$, respectively, and let  $S\subset G,T\subset H$ be finite symmetric sets of generators of $G,H$, respectively. Equip $G\times H$ with the product
word metric. Then,  $\mathbb{X}_n$ embeds isometrically into  $(\Z_2\wr (G\times H))_0$, as seen by taking geodesic segments $I=\{g_0=e_{G},g_1,\ldots,g_{2n}\}\subset G$ and $J=\{h_0=e_H,h_1,\ldots,h_{2n}\}\subset H$ in the Cayley graphs of $G$ and $H$, respectively (those exist since $|G|=|H|=\infty$), and considering the following subset of $\Z_2\wr (G\times H)$: 
$$
\Big\{\big(A,(e_G,e_H)\big)\in \Z_2\wr (G\times H): A\subset I\times J\Big\},
$$ which is  isometric to $\mathbb{X}_n$ (to check this, note that every walk in     $G\times H$ can be retracted in a $1$-Lipschitz manner to $I\times J$ by the mapping that assigns $(g_{\min\{d_G(e_G,x),2n\}},h_{\min\{d_H(e_H,y),2n\}})$ to each $(x,y)\in G\times H$).  By Lemma~\ref{lem:embed in each}, this shows that the first part of Theorem~\ref{thm:MainIntro} holds with $\Z_2\wr \Z^2$ replaced by $\Z_2\wr (G\times H)$. 
\end{remark}

\section{Proof of Theorem~\ref{thm:LpcompressionIntro}}\label{sec:new zer section compression}

Here we will prove the following result, which includes Theorem~\ref{thm:LpcompressionIntro} as a special case:

\begin{theorem}\label{thm:poly growth version} Suppose that $G$ is a group that is generated by some finite symmetric $S\subset G$, and let $d_S$ be the associated left-invariant word metric on $G$. Letting  $B_{d_S}(x,r)=\{y\in G:\ d_S(x,y)\le r\}$ denote the closed $d$-ball of radius $r\in \R$ centered at $x\in G$ (thus, $B_{d_S}(x,r)=\{x\}$ if $0\le r<1$ and $B_{d_S}(x,r)=\emptyset$ if $r<0$), and letting  $e_G$ be the identity element of $G$, assume that $m\in \N$ and $0<c\le C$ satisfy: 
\begin{equation}\label{eq:poly growth assumption}
\forall r\in \N,\qquad c  r^m\le |B_{d_S}(e_G,r)|\le C r^m.
\end{equation}
Then, for every $1\le q<\infty$ we have:
\begin{equation}\label{eq:proposed-exponentintro2}
\alpha_q^*\big((\Z_2\wr G)_0\big)=\max\left\{\frac{m+q-1}{qm},\frac{m+1}{2m}\right\}.
\end{equation}
\end{theorem}

We will henceforth work only in the setting of Theorem~\ref{thm:poly growth version}, so we can simplify notations by writing $e=e_G$ and $d=d_S$, and also supress the subscript $d_S$ from the notation for closed balls, as well as the super scripts $S$ and $d_S$ in the notations~\eqref{eq:lamp to TSP} for the TSP cost and the word metric (with respect to the standard generating set) on $\Z_2\wr G$  (no other metric, group, or generating set will occur, so this does not lead to ambiguity).  We will also abuse notation  by letting $\lesssim_G, \gtrsim_G,\asymp_G$ indicate the corresponding asymptotic statements holding up to constants which may depend only on $c,C,m$. 

The ensuing proof of Theorem~\ref{thm:poly growth version} builds heavily on~\cite{NaorPeres2011}, with some new ingredients. We will first explain why the upper bound on $\alpha_q^*((\Z_2\wr G)_0)$ in Theorem~\ref{thm:poly growth version} holds. If $m=1$, then there is nothing to prove because the right hand side of~\eqref{eq:proposed-exponentintro2} equals $1$. In~\cite{NaorPeres2008}, the upper bound on $\alpha_q^*((\Z\wr\Z)_0)$ in~\eqref{eq:zer0 section} is established using an observation of~\cite{ArzhantsevaGubaSapir2006} that for every $n\in \N$, the zero section of $\Z\wr\Z$ contains with distortion $O(1)$ copies of $n$-dimensional Hamming cubes $(\{0,1\}^n,\|-\cdot-\|_1)$ with their metric rescaled by $n$.  It turns out  that if $m\ge 2$, then an analogous embedding exists into the zero section of lamplighter over $G$ in which the lamps can only be on  or off; if $G=\Z^m$, then such an embedding is exhibited by assigning every $x\in \{0,1\}^{\{1,\ldots,n\}^{m-1}}$ to  $(\{1,\ldots,n\}\times \supp(x),0)\in (\Z_2\wr \Z^m)_0$. In the general setting of  Theorem~\ref{thm:poly growth version}, this can be deduced from the recent work~\cite{GartlandRandrianantoaninaRandrianarivony2026}, as explained in the following lemma:

\begin{lemma}\label{lem:hamming-copy} Suppose that we are in the setting of Theorem~\ref{thm:poly growth version}, except that we only require the first inequality in~\eqref{eq:poly growth assumption}.  For every $n\in \N$ there is $M\in \N$ with $M\asymp  c n^{m-1}$ and 
 $\Psi_n:\{0,1\}^{{M}}\longrightarrow (\Z_2\wr G)_0$ satisfying:
\begin{equation}\label{eq:hamming-rescaled}
 \forall \{x,y\}\in \{0,1\}^{M},\qquad d_{\Z_2\wr G}\big(\Psi_n(x),\Psi_n(y)\big)\asymp n \|x-y\|_1. 
\end{equation}
\end{lemma}

\begin{proof} Using the terminology of~\cite[Definition~3.2]{GartlandRandrianantoaninaRandrianarivony2026}, there is an integer $K\asymp c n^{m-1}$ such that $G$ is not $K$-TSP-efficient, i.e., there are $U,V\subset G$ and $x,y\in G$ for which $\TSP((U,x),(V,y))>K\diam(\{x,y\}\cup U\triangle V)$, where $\diam(\cdot)$ denotes the diameter with respect to $d$. Indeed,  take $x=y=e$  as well as $(U,x)=(B(e,n),e)$ and $(V,y)=(\emptyset,e)$, so $\TSP(B(e,n);e,e)\gtrsim c n^m$ by~\eqref{eq:TSP to support size} and~\eqref{eq:poly growth assumption}, while $\diam(B(e,n))\le  2n$. A verbatim instantiation  of the construction in the proof of~\cite[Theorem~3.9]{GartlandRandrianantoaninaRandrianarivony2026} to these witnesses to the failure of $K$-TSP-efficiency yields an integer $M\asymp K\gtrsim c n^{m-1}$ and pairwise disjoint  $A_1,\ldots,A_M\subset B(e,n)$ such that: 
\begin{equation}\label{eq:quote not efficient conclusion}
\forall J\subset \{1,\ldots,M\},\qquad \TSP \bigg(\bigcup_{j\in J} A_j,\emptyset \bigg)\asymp n|J|. 
\end{equation}
We can therefore define:
$$
\forall x\in \{0,1\}^M,\qquad \Psi_n(x)\eqdef \bigg(\bigcup_{j\in \supp(x)} A_j,e\bigg),
$$
so that the desired conclusion~\eqref{eq:hamming-rescaled} follows from~\eqref{eq:quote not efficient conclusion} using~\eqref{eq:lamp to TSP zero section}.   
\end{proof}

By~\cite{Enflo69,NaorSchechtmanNonlinearType}, for every $M\in \N$, every $\f:\Z_2^M\to L_q$   satisfies: 
\begin{equation}\label{eq:enflo}
 \Big(\E\big[\|\f(x+e_1+\ldots+e_M)-\f(x)\|_q^{\min \{q,2\}}\big]\Big)^{\frac{1}{\min\{q,2\}}}\lesssim \sqrt{q}\Big(\sum_{i=1}^M \E\big[\|\f(x+e_i)-\f(x)\|_q^{\min\{q,2\}}\big]\Big)^{\frac{1}{\min\{q,2\}}},
\end{equation}
where the expectations are with respect to $x\in \Z_2^M$ distributed uniformly at random.

\begin{proof}[Proof of the upper bound on $\alpha_q^*((\Z_2\wr G)_0)$ in Theorem~\ref{thm:poly growth version}] Let $f:(\Z_2\wr G)_0\to L_q$ be $1$-Lipschitz and satisfy $\|f(x)-f(y)\|_q\ge \kappa d_{\Z_2\wr G}(x,y)^\alpha$ for all $x,y\in (\Z_2\wr G)_0$ and some $\alpha,\kappa>0$. Fix $n\in \N$.  Set $\f_n=f\circ \Psi_n$, where $\Psi_n$ is as in  Lemma~\ref{lem:hamming-copy}. An application of~\eqref{eq:enflo} to $\f_n$ with the aforementioned distances guarantees (namely, the bi-Lipschitz condition for $\Psi_n$ and the compression assumption on $f$) gives the asymptotic estimate $(nM)^\alpha\lesssim \sqrt{q}nM^{1/\min\{q,2\}}$, i.e., $M^{\alpha-1/\min\{q,2\}} \lesssim \sqrt{q}n^{1-\alpha}$. If  $\alpha\le 1/\min\{q,2\}$, then a fortiori $\alpha$ is at most the right hand side of~\eqref{eq:proposed-exponentintro2}.  If $\alpha> 1/\min\{q,2\}$, then using the lower bound $M\gtrsim cn^{m-1}$ in Lemma~\ref{lem:hamming-copy} we get  $c^{\alpha-1/\min\{q,2\}} n^{(m-1)(\alpha-1/\min\{q,2\})}\lesssim \sqrt{q} n^{1-\alpha}$. This forces $(m-1)(\alpha-1/\min\{q,2\})\le 1-\alpha$ by letting $n\to \infty$, which simplifies to show that $\alpha$ is at most the right hand side of~\eqref{eq:proposed-exponentintro2}, as desired.  
\end{proof}

We will next prove the rest of Theorem~\ref{thm:poly growth version}, i.e., build the desired embedding. One difference from the reasoning in~\cite{NaorPeres2011} is that in the present higher dimensional setting we encode multi-scale structure of subsets of $G$ through  ``snapshots''   that are ``doubly local,'' i.e., they are simultaneous intersections with both a ball and an annular region. Specifically, given    $A\subset G$ and $k,s\in \N$ define $A_{k,s,x}\subset A$ by:
\begin{equation}\label{def:our snapshot}
A_{k,s,x}\eqdef A\cap \big(B(e,2^{k-1}-1)\setminus B(e,2^{k-2}-1)\big)\cap B(x,2^{s-2}).
\end{equation}
The indexing in~\eqref{def:our snapshot} is such that $B(e,2^{k-1}-1)\setminus B(e,2^{k-2}-1)=\{e\}$ if $k=1$ and $B(x,2^{s-2})=\{x\}$ if $s=1$. 
See Figure~\ref{fig:local-snapshot} below for a schematic depiction of~\eqref{def:our snapshot}. 
\begin{figure}[H]
  \centering
  \includegraphics[width=.78\linewidth]
    {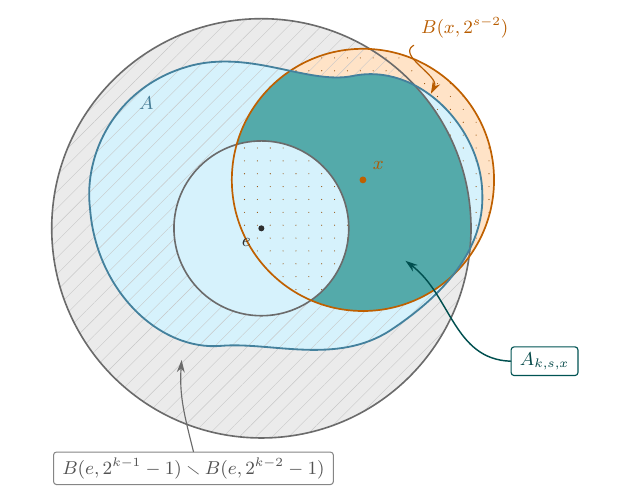}
  \caption{For $A\subset G$, two scales $k,s\in \N$, and a point $x\in G$, the region $A_{k,s,x}\subset G$  consists of those elements $y$ of $A$ that satisfy both $d(y,x)\le 2^{s-2}$ and $2^{k-2}-1< d(y,e)<2^{k-1}$. }
  \label{fig:local-snapshot}
\end{figure}
The following simple lemma records basic properties of~\eqref{def:our snapshot} that will be used later:

\begin{lemma}\label{lem:basic properties of snapshot that do not involve net} Fix  $A\subset G$ and $x\in G$. Suppose that  $k,s\in \N$ satisfy $k\ge s$ and  $A_{k,s,x}\neq \emptyset$. Then,
\begin{equation}\label{eq:two basic facts of snapshot}
\TSP(A;e,e)+1> 2^{k-2}\qquad\mathrm{and}\qquad \left\{\ell \in \{k,k+1,\ldots\}:\ A_{\ell,s,x}\neq \emptyset\right\}\subset \{k,k+1\}. 
\end{equation}
\end{lemma}
\begin{proof} Consider any  $y\in A_{k,s,x}$. By~\eqref{def:our snapshot} we know that  $y\in A$ and  $d(y,e)\ge \lceil 2^{k-2}-1\rceil$. As any TSP tour that covers  $A$ visits $y$, we get $\TSP(A;e,e)\ge 2\lceil 2^{k-2}-1\rceil$, so the first part of~\eqref{eq:two basic facts of snapshot} holds.   Also, from $y\in A_{k,s,x}$ we  get that $d(x,y)\le 2^{s-2}$ and $2^{k-2}-1<d(y,e)\le 2^{k-1}-1$. If  $\ell\in\{k+2,k+3,\ldots\}$ and $z\in G\setminus B(e,2^{\ell-2}-1)$, then $d(x,z)\ge d(z,e)-d(y,e)-d(x,y)>2^{\ell-2}-1-(2^{k-1}-1)-2^{s-2}\ge 2^k-2^{k-1}-2^{s-2}\ge 2^{s-2}$, i.e., $z\notin B(x,2^{s-2})$. This shows that $A_{\ell,s,x}=\emptyset$ for every  $\ell\in\{k+2,k+3,\ldots\}$, i.e., the second part of~\eqref{eq:two basic facts of snapshot} holds as well. 
\end{proof}

Observe that second part of~\eqref{eq:two basic facts of snapshot} implies that  the following multiplicity restriction holds:
\begin{equation*}
\forall A\subset G,\ \forall (x,s)\in G\times \N,\qquad \big|\{k\in \{s,s+1,\ldots\}:\ A_{k,s,x}\neq \emptyset\}\big|\le 2. 
\end{equation*}
Consequently, because~\eqref{def:our snapshot} implies that   $A_{k,s,x}\subset A\cap B(x,2^{s-2})$ for every $A\subset G$ and $k,s\in \N$, we get: 
\begin{equation}\label{eq:multiplicity 2}
\forall A,\sub\subset G,\  \forall s\in \N,\qquad \big|\{(k,x)\in \{s,s+1,\ldots\}\times \sub :\ A_{k,s,x}\neq\emptyset\}\big| \le 2\big|\{x\in \sub:\ A\cap B(x,2^{s-2})\neq\emptyset\} \big|.
\end{equation}

For each $s\in \N$ fix a subset $\NN_s\subset G$ that is maximal with respect to inclusion relative to the requirement that $d(x,y)>2^{s-2}$ for every distinct $x,y\in \NN_s$. We thus have:
\begin{equation}\label{eq:def Ns}
 \min_{\substack{x,y\in \NN_s\\ x\neq y}} d(x,y)>2^{s-2}\qquad \mathrm{and}\qquad G=\bigcup_{x\in \NN_s} B(x,2^{s-2}).
\end{equation}
Observe that this implies that $\NN_1=G$. Furthermore,  a standard packing argument shows that thanks to the first part of~\eqref{eq:def Ns} and the  polynomial growth assumption~\eqref{eq:poly growth assumption}, we have:
\begin{equation}\label{eq:net growth}
\forall x\in G,\ \forall \sigma>0,\qquad \big|\NN_s\cap B(x,\sigma 2^{s})\big|\lesssim_G (\sigma+1)^m.
\end{equation}
(Briefly, note that by the triangle inequality for $d$ the first part of~\eqref{eq:def Ns} implies that the balls $\{B(y,2^{s-3})\}_{y\in \NN_s}$ are pairwise disjoint, and also  $\bigcup_{y\in \NN_s\cap B(x,\sigma2^{s})}B(y,2^{s-3})\subset B(x,(\sigma+1/8) 2^{s})$. Now use~\eqref{eq:poly growth assumption}.) 

For each  $A\subset G$, we will  consider the following counts of its presence at various scales and locations: 
\begin{equation}\label{eq:def our counts}
\forall k,s\in \N,\qquad C_{k,s}(A)\eqdef \big|\{x\in \NN_s:\ A_{k,s,x}\neq \emptyset \}\big|.
\end{equation}
Then, the following a priori bounds hold for every $A\subset G$:
\begin{equation}\label{eq:a priori count growth}
\forall k,s\in \N,\qquad  C_{k,s}(A)\lesssim_G2^{m(k-s)}+1. 
\end{equation}
Indeed, fix $k,s\in \N$ and  $x\in G$ with $A_{k,s,x}\neq \emptyset$.  Then, $d(x,e)\le d(x,y)+d(y,e)<2^{s-2}+2^{k-1}$ for any $y\in A_{k,s,x}$, so $\{x\in \NN_s:\ A_{k,s,x}\neq \emptyset \}\subset \NN_s\cap B(e,2^{s-2}+2^{k-1})$, whence~\eqref{eq:a priori count growth} follows from~\eqref{eq:net growth} and~\eqref{eq:def our counts}.

\begin{lemma} For every finite subset $A$ of $G$ and every $\beta\ge 0$ we have: 
\begin{equation}\label{eq:relate the counts to TSP}
\TSP(A;e,e)\lesssim_G \sum_{k=1}^{\infty}\sum_{s=1}^k 2^sC_{k,s}(A)\qquad\mathrm{and}\qquad \forall s\in \N,\qquad \sum_{k=s}^{\infty} 2^{\beta k}C_{k,s}(A)\lesssim_{G} \frac{4^\beta}{2^s}\big(\TSP(A;e,e)+1\big)^{\beta+1}.
\end{equation}
\end{lemma}

\begin{proof}

If $E\subset G$ is finite, then for  every $r\ge 0$ let $N(E,r)$ denote the minimum $n\in \N$ for which there exist $x_1,\ldots, x_n\in G$ such that $E\subset B(x_1,r)\cup\ldots\cup B(x_n,r)$. By~\cite[equation~(17)]{NaorPeres2011}  we have:
\begin{equation}\label{eq:equation 17 bound quote}
\forall E\subset B(e,2^k),\qquad \TSP(E;e,e)\lesssim \sum_{j=0}^{k} 2^j N(E,2^{j-1}). 
\end{equation}
Consequently, for every $k\in \N$ we have: 
\begin{equation}\label{eq:cover annulus}
\TSP\Big(A\cap \big(B(e,2^{k-1}-1)\setminus B(e,2^{k-2}-1)\big);e,e\Big)\lesssim \sum_{s=1}^{k} 2^s N\Big(A\cap \big(B(e,2^{k-1}-1)\setminus B(e,2^{k-2}-1)\big),2^{s-2}\Big).
\end{equation}

To apply~\eqref{eq:cover annulus}, observe that thanks to~\eqref{def:our snapshot} and  the second part of~\eqref{eq:def Ns} for every $k,s\in \N$ we have:
$$
A\cap \big(B(e,2^{k-1}-1)\setminus B(e,2^{k-2}-1)\big)\subset \bigcup_{\substack{x\in \NN_s\\ A_{k,s,x}\neq\emptyset}} B(x,2^{s-2}).
$$
Hence, $N(A\cap (B(e,2^{k-1}-1)\setminus B(e,2^{k-2}-1)),2^{s-2})\le C_{k,s}(A)$ by~\eqref{eq:def our counts}. A substitution of this into~\eqref{eq:cover annulus} gives:
\begin{equation}\label{eq:TSP of annulus}
\forall k\in \N,\qquad \TSP\Big(A\cap \big(B(e,2^{k-1}-1)\setminus B(e,2^{k-2}-1)\big);e,e\Big)\lesssim \sum_{s=1}^{k} 2^s C_{k,s}(A).
\end{equation}
Now, consider for every $k\in \N$  the optimal TSP tour that starts at $e$, covers $A\cap (B(e,2^{k-1}-1)\setminus B(e,2^{k-2}-1))$, and returns to $e$. By concatenating these walks, we get the following trivial upper bound on  $\TSP(A;e,e)$:
\begin{equation}\label{eq:TSP concatenation}
\TSP(A;e,e)\le \sum_{k=1}^\infty \TSP\Big(A\cap \big(B(e,2^{k-1}-1)\setminus B(e,2^{k-2}-1)\big);e,e\Big). 
\end{equation} 
The first part of~\eqref{eq:relate the counts to TSP} is a substitution of~\eqref{eq:TSP of annulus} into~\eqref{eq:TSP concatenation}.

We will next explain why the second part of~\eqref{eq:relate the counts to TSP} holds. Fix $s\in \N$. If $C_{k,s}(A)=0$ for every  $k\in \{s,s+1,\ldots\}$, then there is nothing to prove, so we may assume   from now that: 
\begin{equation}\label{eq:def Ks}
\mathbb{K}_s\eqdef \big\{k\in \{s,s+1,\ldots\}:\ C_{k,s}(A)\neq 0\big\}\neq \emptyset.
\end{equation}
If $k\in \mathbb{K}_s$, then $C_{k,s}(A)\neq 0$, so by~\eqref{eq:def our counts} there is $x\in \NN_s$ with $A_{k,s,x}\neq \emptyset$, whence    Lemma~\ref{lem:basic properties of snapshot that do not involve net} gives:
\begin{equation}\label{eq:TSP a priori lower s}
2^s\le 2^{\max \mathbb{K}_s}\le 4\big(\TSP(A;e,e)+1\big).
\end{equation}
Consequently, the sum that we need to bound for the second part of~\eqref{eq:relate the counts to TSP} satisfies:
\begin{align}\label{eq:take alpha out}
\begin{split}
\sum_{k=s}^{\infty} 2^{\beta k}C_{k,s}(A)&\stackrel{\eqref{eq:def Ks}}{=}\sum_{k\in \mathbb{K}_s} 2^{\beta k}C_{k,s}(A)\\&\stackrel{\eqref{eq:TSP a priori lower s}}{\le} 4^\beta \big(\TSP(A;e,e)+1\big)^\beta \sum_{k=s}^\infty C_{k,s}(A)\\&\stackrel{\eqref{eq:def our counts}}{=} 4^\beta \big(\TSP(A;e,e)+1\big)^\beta\big|\{(k,x)\in \{s,s+1,\ldots\}\times \NN_s:\ A_{k,s,x}\neq\emptyset\} \big|\\&\stackrel{\eqref{eq:multiplicity 2}}{\lesssim} 4^\beta \big(\TSP(A;e,e)+1\big)^\beta\big|\{x\in \NN_s:\ A\cap B(x,2^{s-2})\neq\emptyset\} \big|.
\end{split}
\end{align}

Set $\ell \eqdef \TSP(A;e,e)$. Fix $y_0,\ldots,y_\ell\in G$  such that $y_{i-1}^{-1}y_{i}\in S$ for every $i\in \{1,\ldots,\ell\}$, as well as $y_0=y_\ell=e$ and $A\subset \{y_0,\ldots,y_\ell\}$. For  $x\in \NN_s$ with $A\cap B(x,2^{s-2})\neq\emptyset$, fix $i(x)\in \{0,\ldots,\ell\}$ for which $d(x,y_{i(x)})\le 2^{s-2}$. Let $j(x)\in \{0,\ldots, \lfloor \ell/2^s\rfloor\}$ be the unique such index with  $j2^s \le i(x)<(j+1)2^s$. Then: $$d(x,y_{j(x)2^s})\le d(x,y_{i(x)})+\sum_{r=j(x)2^s}^{i(x)-1} d(y_{r},y_{r+1})\le 2^{s-2}+i(x)-j(x)2^s<2^{s-2}+2^s<2^{s+1},$$
i.e., $x\in B(y_{j(x)2^s},2^{s+1})$. Hence, $\{x\in \NN_s:\ A\cap B(x,2^{s-2})\neq\emptyset\ \wedge \ j(x)=j\}\subset \NN_s\cap B(y_{j2^s},2^{s+1})$ for every $j\in \{0,\ldots, \lfloor \ell/2^s\rfloor\}$. We therefore get the following estimate, which implies the rest of~\eqref{eq:relate the counts to TSP} thanks to~\eqref{eq:take alpha out}:
\begin{align*}
\big|\{x\in \NN_s:\ &A\cap B(x,2^{s-2})\neq\emptyset\} \big|=\sum_{j=0}^{\left\lfloor \frac{\ell}{2^s}\right\rfloor} \big|\{x\in \NN_s:\ A\cap B(x,2^{s-2})\neq\emptyset\ \wedge \ j(x)=j\} \big|\\&\le  \sum_{j=0}^{\left\lfloor \frac{\ell}{2^s}\right\rfloor}\big|\NN_s\cap B(y_{j2^s},2^{s+1})\big|\stackrel{\eqref{eq:net growth}}{\lesssim_G}1+\left\lfloor \frac{\ell}{2^s}\right\rfloor= 1+\left\lfloor \frac{\TSP(A;e,e)}{2^s}\right\rfloor\stackrel{\eqref{eq:TSP a priori lower s}}{\lesssim}\frac{\TSP(A;e,e)+1}{2^s}.\tag*{\qedhere}
\end{align*}
\end{proof}

\begin{proof}[Proof of the  lower bound on $\alpha_q^*((\Z_2\wr G)_0)$ in Theorem~\ref{thm:poly growth version}]  It suffices to assume that $1\le q\le 2$, as  $L_q$ contains $L_2$ isometrically (see e.g.~\cite[Proposition 6.4.2]{AlbiacKalton}). It is convenient to index the standard basis of $\ell_q$ by $e_{k,s,x}(U)$ as $(k,s,x,U)$ ranges over all possible $(k,s,x)\in \N\times \N\times G$ and $U$ ranges over all  possible finite subsets $U$ of $G$. With this notation, define $f:(\Z_2\wr G)_0\to \ell_q$ as follows:
\begin{equation}\label{eq:out higher dim NP}
\forall (A,e)\in (\Z_2\wr G)_0,\qquad f(A,e)\eqdef \sum_{k=1}^\infty 2^{\frac{q-1}{q}k}\sum_{s=1}^k \frac{2^{\frac{(m-1)(2-q)+1}{qm}s}}{s^{\frac{2}{q}}}  \sum_{x\in \NN_s} \big(e_{k,s,x}(A_{k,s,x})-e_{k,s,x}(\emptyset) \big). 
\end{equation}
Observe that only  finitely many of the summands in~\eqref{eq:out higher dim NP} are nonzero. Indeed,  as $A$ is finite, there are only finitely many $k\in \N$ for which $A$ has nonempty intersection with the annulus $B(e,2^{k-1}-1)\setminus B(e,2^{k-2}-1)$. Consequently, we need to check that for fixed $k\in \N$ and $s\in \{1,\ldots,k\}$ there are only finitely many $x\in \NN_s$ for which $A_{k,s,x}\neq \emptyset$. By~\eqref{def:our snapshot}, any such $x$ belongs to  $\bigcup_{a\in A} \NN_s\cap B(a,2^{s-2})$, which is finite by~\eqref{eq:net growth} as $|A|<\infty$.

If $U,V\subset G$ are finite and distinct, then:  
\begin{align}\label{eq:distance UV}
\begin{split}
\|f(U,e)-f(V,e)\|_q&\stackrel{\eqref{eq:def our counts}\wedge \eqref{eq:out higher dim NP}}{\asymp}\bigg(\sum_{s=1}^\infty  \frac{2^{\frac{(m-1)(2-q)+1}{m}s}}{s^2}\sum_{k=s}^\infty  2^{(q-1)k}C_{k,s}(U\triangle V)\bigg)^{\frac{1}{q}}\\&\stackrel{\eqref{eq:relate the counts to TSP}}{\lesssim}
\bigg(\sum_{s=1}^\infty \frac{ 2^{\frac{(m-1)(2-q)+1}{m}s}}{s^22^s}\bigg)^{\frac{1}{q}}\big(\TSP(U\triangle V;e,e)+1\big) \stackrel{\eqref{eq:lamp to TSP zero section}}{\asymp} d_{\Z_2\wr G}\big((U,e),(V,e)\big),
\end{split}
\end{align}
where the last step of~\eqref{eq:distance UV} uses $\frac{(m-1)(2-q)+1}{m}s-s=-\frac{m-1}{m}(q-1)s\le 0$. Thus, $f$ is $O(1)$-Lipschitz. 

To lower bound $\|f(U,e)-f(V,e)\|_q$, take the first equality in~\eqref{eq:distance UV}  and examine the largest summand:
\begin{equation}\label{eq:sum to sup}
\|f(U,e)-f(V,e)\|_q\stackrel{\eqref{eq:distance UV}}{\gtrsim} \sup_{\substack{k,s\in \N\\  s\le k}}  \frac{2^{\frac{(m-1)(2-q)+1}{qm}s+\frac{q-1}{q}k}}{s^{\frac{2}{q}}} C_{k,s}(U\triangle V)^{\frac{1}{q}}\stackrel{\eqref{eq:a priori count growth}}{\gtrsim_G} \sup_{\substack{k,s\in \N\\  s\le k}} \bigg(\frac{2^sC_{k,s}(U\triangle V)}{s^{\frac{2m}{m+q-1}}}\bigg)^{\frac{m+q-1}{qm}}.
\end{equation}
Denote $R=R_{U\triangle V}\eqdef\min\{r\in \{2,3,\ldots\}:\ U\triangle V\subset B(e,r)\}$. There is $y\in U\triangle V$ with $d(e,y)\ge R\1_{\{R\ge 3\}}$ (the indicator was inserted as if $U\triangle V\subset B(e,2)$, then $R=2$), so $\TSP(U\triangle V;e,e)\ge 2R\1_{\{R\ge 3\}}$ since any tour that starts and ends at $e$ and visits all of $U\triangle V$  must in particular visit $y$. This observation gives:
\begin{equation}\label{eq:dist at least R+1}
\log R\asymp \log(2R\1_{\{R\ge 3\}}+1)+1\stackrel{\eqref{eq:lamp to TSP zero section}}{\le} \log d_{\Z_2\wr G}\big((U,e),(V,e)\big)+1. 
\end{equation}
By~\eqref{def:our snapshot}, if $k> 3\lceil \log_2 R\rceil$, whence $2^{k-2}-1\ge R$, then  $(U\triangle V)_{k,s,x}=\emptyset$ for every $s\in \N$ and $x\in G$. Hence: 
\begin{align}\label{eq:use upper TSP bound}
\begin{split}
&d_{\Z_2\wr G}\big((U,e),(V,e)\big)\stackrel{\eqref{eq:lamp to TSP zero section}}{\asymp} \TSP(U\triangle V;e,e)+1\stackrel{\eqref{eq:relate the counts to TSP}}{\lesssim_G} \sum_{k=1}^{3\lceil \log_2 R\rceil}\sum_{s=1}^k 2^sC_{k,s}(U\triangle V)\\&\ \lesssim  \sum_{k=1}^{3\lceil \log_2R\rceil}(\log R)^{\frac{2m}{m+q-1}}\sum_{s=1}^k \frac{2^sC_{k,s}(U\triangle V)}{s^{\frac{2m}{m+q-1}}}\stackrel{\eqref{eq:dist at least R+1}}{\lesssim} \big(\log d_{\Z_2\wr G}\big((U,e),(V,e)\big)+1\big)^{\frac{2m}{m+q-1}+2} \sup_{\substack{k,s\in \N\\  s\le k}} \frac{2^sC_{k,s}(U\triangle V)}{s^{\frac{2m}{m+q-1}}}.
\end{split}
\end{align}
Note that the second step of~\eqref{eq:use upper TSP bound} is verbatim the first part of~\eqref{eq:relate the counts to TSP} if $\TSP(U\triangle V;e,e)\ge 1$, so the only point to check  is that~\eqref{eq:use upper TSP bound}  also holds when $U\triangle V=\{e\}$ (recall that we are assuming $U\neq V$), but then $R=2$ and $C_{1,1}(\{e\})=1$, as   $e\in \{e\}_{1,1,e}$ and $e\in \NN_1=G$. The  lower bound on $\alpha_q^*((\Z_2\wr G)_0)$ in Theorem~\ref{thm:poly growth version} follows by substituting~\eqref{eq:use upper TSP bound} into~\eqref{eq:sum to sup}.   
\end{proof}

\bibliographystyle{alphaabbrvprelim}
\bibliography{biblio.bib}

\end{document}